\documentclass[12 pt,twoside]{amsart}

\usepackage{amsopn}
\usepackage{amssymb}
\usepackage{amscd}

\newtheorem{theorem}{Theorem}[section]
\newtheorem{lemma}[theorem]{Lemma}

\newtheorem{corollary}[theorem]{Corollary}

\theoremstyle{definition}

\theoremstyle{remark}
\newtheorem{remark}[theorem]{Remark}

\numberwithin{equation}{section}

\begin{document}

\title[Isometric actions on a locally compact metric space]{On the action of the group of isometries on a locally compact metric space}

\author[Antonios Manoussos]{Antonios Manoussos}
\address{Fakult\"{a}t f\"{u}r Mathematik, SFB 701, Universit\"{a}t Bielefeld, Postfach 100131, D-33501 Bielefeld, Germany}
\email{amanouss@math.uni-bielefeld.de}
\thanks{During this research the author was fully supported by SFB 701 ``Spektrale Strukturen und Topologische Methoden in der Mathematik" at the University of Bielefeld,
Germany. He would also like to express his gratitude to Professor H. Abels for his support.}


\date{}


\begin{abstract}
In this short note we give an answer to the following question. Let $X$ be a locally compact metric space with group of isometries $G$. Let $\{g_i\}$ be a net in $G$
for which $g_ix$ converges to $y$, for some $x,y\in X$. What can we say about the convergence of $\{ g_i\}$? We show that there exist a subnet $\{g_j\}$ of $\{g_i\}$
and an isometry $f:C_x\to X$  such that $g_{j}$ converges to $f$ pointwise on $C_x$ and $f(C_x)=C_{f(x)}$, where $C_x$ and $C_y$ denote the pseudo-components of $x$
and $y$ respectively. Applying this we give short proofs of the van Dantzig--van der Waerden theorem (1928) and Gao--Kechris theorem (2003).
\end{abstract}

\maketitle

\section{The main result and some applications}
A few words about the notation we shall be using. In what follows, $X$ will denote a locally compact metric space with group of isometries $G$. If we endow $G$ with
the topology of pointwise convergence then $G$ is a topological group \cite[Ch. X, \S 3.5 Corollary]{bour2}. On $G$ there is also the topology of uniform convergence
on compact subsets which is the same as the compact-open topology. In the case of a group of isometries these topologies coincide with the topology of pointwise
convergence, and the natural action of $G$ on $X$ with $(g,x)\mapsto g(x)$, $g\in G$, $x\in X$, is continuous \cite[Ch. X, \S 2.4 Theorem 1 and \S 3.4 Corollary
1]{bour2}. For $F\subset G$, let $K(F):=\{ x\in X\, |\,\,\mbox{the set}\,\, Fx\,\,\mbox{has compact closure in}\,\, X\}$. The sets $K(F)$ are clopen \cite[Lemma
3.1]{manstra}.

\begin{lemma}\label{lemma1}
Let $\Gamma =\{g_i\}$ be a net in $G$ and $x\in K(\Gamma)$ such that $g_ix$ converges to $y$ for some $y\in X$. Then a subnet of $\Gamma$ converges to an isometry
$f:K(\Gamma)\to X$ on $K(\Gamma)$.
\end{lemma}
\begin{proof}
Let $g_i|_{K(\Gamma)}$ denote the restriction of $g_i$ on $K(\Gamma)$. Arzela--Ascoli theorem implies that the set $\{ g_i|_{K(\Gamma)}:K(\Gamma)\to X \}$ has compact
closure in the set of all continuous maps from $K(\Gamma)$ to $X$. Thus, there exist a subnet $\{g_j\}$ of $\{g_i\}$ and an isometry $f:K(\Gamma)\to X$ such that
$g_j\to f$ on $K(\Gamma)$.
\end{proof}

In \cite{kechris} S. Gao and A. S. Kechris introduced the concept of pseudo-components. These are the equivalence classes $C_x$ of the following equivalence relation:
$x\sim y$ if and only if x and y, as also y and x, can be connected by a finite sequence of intersecting open balls with compact closure. The pseudo-components are
clopen \cite[Proposition 5.3]{kechris}. We call $X$ pseudo-connected if it has only one pseudo-component. An immediate consequence of the definitions is that
$gC_x=C_{gx}$ for every $g\in G$. Another notion, that will be used in the proofs, is the radius of compactness $\rho (x)$ of $x\in X$ \cite{kechris}. Let $B_r(x)$
denote the open ball centered at $x$ with radius $r>0$. Then  $\rho (x):=\sup \{\, r>0\, |\,\, B_r(x)\,\,\mbox{ has compact closure}\}$. If $\rho (x)=+\infty$ for
some $x\in X$ then every ball has compact closure (i.e., $X$ has the Heine--Borel property), hence $\rho (x)=+\infty$ for every $x\in X$. If $\rho (x)$ is finite for
some $x\in X$ then the radius of compactness is a Lipschitz map \cite[Proposition 5.1]{kechris}. Note that $\rho$ is $G$-invariant.

\begin{lemma}\label{lemma2}
Let $x, y\in X$ and $\{ g_i\}_I$ be a net in $G$ with $g_ix\to y$. Then there is an index $i_0\in I$ such that $C_x\subset K(F)$, where $F:=\{ g_i\,|\,i\geq i_0\}$.
\end{lemma}
\begin{proof}
Since $X$ is locally compact there exists an index $i_0$ such that the set $F(x)$ has compact closure, where $F:=\{ g_i\,|\,i\geq i_0\}$. We claim that for every
$z\in C_x$ the set $F(z)$ also has compact closure, hence $C_x\subset K(F)$. The strategy is to start with an open ball $B_r(x)$ with radius $r<\rho (x)$ and prove
that $F(z)$ has compact closure for every $z\in B_r(x)$. Then our claim follows from the definition of $C_x$. To prove the claim take a sequence $\{g_n z\}\subset F$.
Since the closure of $F(x)$ is compact we may assume, upon passing to a subsequence, that $g_nx\to w$ for some $w$ in the closure of $F(x)$. Assume that $\rho (x)$ is
finite and take a positive number $\varepsilon$ such that $r+\varepsilon < \rho (x)$. Then for $n$ big enough
\[
d(g_nz,w)\leq d(g_nz,g_nx) + d(g_nx,w)=d(z,x)+ d(g_nx,w)<r+\varepsilon < \rho (x).
\]
Recall that the radius of convergence is a continuous map, and since $g_nx\to w$ then $ \rho (x)= \rho (w)$. So, the sequence $\{g_n z\}$ is contained eventually in a
ball of $w$ with compact closure, hence it has a convergence subsequence. The same also holds in the case where $\rho (x)=+\infty$.
\end{proof}

\begin{theorem}\label{theorem}
Let $X$ be a locally compact metric space with group of isometries $G$ and let $\{g_i\}$ be a net in $G$ for which $g_ix$ converges to $y$, for some $x,y\in X$. Then
there exist a subnet $\{g_j\}$ of $\{g_i\}$ and an isometry $f:C_x\to X$ such that $g_{j}$ converges to $f$ pointwise on $C_x$ and $f(C_x)=C_{f(x)}$
\end{theorem}

\begin{proof}
By Lemma \ref{lemma2} there is an index $i_0\in I$ such that $C_x\subset K(F)$, where $F:=\{ g_i\,|\,i\geq i_0\}$. Hence, by Lemma \ref{lemma1}, there exists a subnet
$\{ g_j\}$ of $\{ g_i\}$ which converges to an isometry $f:K(F)\to X$ on $K(F)$. Therefore, $g_{j}\to f$ on $C_x$. Let us show that $f(C_x)=C_{f(x)}$. Since
$d(x,g_j^{-1}f(x))=d(g_jx,f(x))\to 0$ it follows that $g_j^{-1}f(x)\to x$. Hence, by repeating the previous procedure, there exist a subnet $\{g_k\}$ of $\{g_j\}$ and
an isometry $h:C_{f(x)}\to X$ such that $g_{k}^{-1}\to h$ pointwise on $C_{f(x)}$ and $h(f(x))=x$. Note that $g_{k}x \in C_{f(x)}$ eventually for every $k$, since
$g_{k}x\to f(x)$ and  $C_{f(x)}$ is clopen. Therefore, $g_kC_x=C_{g_{k}x}=C_{f(x)}$. Take a point $z\in C_x$. Then, $g_{k}z\to f(z)$ and since $C_{f(x)}$ is clopen
then $f(z)\in C_{f(x)}$, so $f(C_x)\subset C_{f(x)}$. By repeating the same arguments as before, it follows that $hC_{f(x)}\subset C_x$. Take now a point $w\in
C_{f(x)}$. Then $h(w)\in C_x$, hence $g_{k}^{-1}(w)\in C_x$ eventually for every $k$. So, $w=g_{k}g_{k}^{-1}(w)\to f(h(w))\in f(C_x)$ from which follows that
$C_{f(x)}\subset f(C_x)$.
\end{proof}

A few words about properness. A continuous action of a topological group $H$ on a topological space $Y$ is called proper (or Bourbaki proper) if the map $H\times Y\to
Y\times Y \,\,\mbox{with}\,\,(g,x)\mapsto(x,gx),\,\,\mbox{for}\,\, g\in H\,\,\mbox{and}\,\, x\in Y$, is proper, i.e., it is continuous, closed and the inverse image
of a singleton is a compact set \cite[Ch. III, \S 4.1 Definition 1]{bour1}. In terms of nets, a continuous action is proper if and only if whenever we have two nets
$\{g_i\}$ in $H$ and $\{x_i\}$ in $Y$, for which both $\{x_i\}$ and $\{g_ix_i\}$ converge, then $\{g_i\}$ has a convergent subnet. For isometric actions, it is easy
to see that a continuous action is proper if and only if whenever we have a net $\{g_i\}$ in $H$ for which $\{g_ix\}$ converges for some $x\in Y$, then $\{g_i\}$ has
a convergent subnet. If $H$ is locally compact and $Y$ is Hausdorff, then $H$ acts properly on $Y$ if and only if for every $x,y\in Y$ there exist neighborhoods $U$
and $V$ of $x$ and $y$, respectively, such that the set $\{ g\in H\, |\, gU\cap V\neq\emptyset\}$ has compact closure in $H$ \cite[Ch. III, \S 4.4 Proposition
7]{bour1}. Observe that if $H$ acts properly on a locally compact space $Y$ then $H$ is also locally compact.

A direct implication of Theorem \ref{theorem} is the van Dantzig--van der Waerden theorem  \cite{d-w}. The advantage of our proof, comparing to the proofs given in
the original work of van Dantzig--van der Waerden or in \cite[Theorem 4.7, pp.~46--49]{kob-nom}, is that it is considerably shorter.

\begin{corollary}\label{cor1}(van Dantzig--van der Waerden theorem 1928)
Let $X$ be a connected locally compact metric space with group of isometries $G$. Then $G$ acts properly on $X$ and is locally compact.
\end{corollary}

Another application of Theorem \ref{theorem} is that we can rederive the results of Gao and Kechris in \cite[Theorem 5.4 and Corollary 6.2]{kechris}.

\begin{corollary}\label{cor2}(Gao--Kechris theorem 2003)
Let $X$ be a locally compact metric space  with finitely many pseudo-components. Then the group of isometries $G$ of $X$ is locally compact. If $X$ is
pseudo-connected, then $G$ acts properly on $X$.
\end{corollary}
\begin{proof}
Let $C_1, C_2, \ldots, C_n$ denote the pseudo-components of $X$ and take points $x_1\in C_1, x_2\in C_2, \ldots, x_n\in C_n$ and open balls $B_r(x_m)\subset C_m$,
$m=1,2,\ldots, n$, $r>0$ such that all $B_r(x_m)$ have compact closures. We will show that the set $V:=\bigcap_{m=1}^{n}\{ g\in G\,|\, gx_m\in B_r(x_m)\}$ is an open
neighborhood of the identity in $G$ with compact closure. Indeed, take a net $\{ g_i\}$ in $V$. Since each $B_r(x_m)$ has compact closure there exist a subnet $\{
g_j\}$ of $\{ g_i\}$ and points $y_1\in C_1, y_2\in C_2, \ldots, y_n\in C_n$ such that $g_jx_m\to y_m$ for every $m=1,2,\ldots, n$. Theorem \ref{theorem} implies that
there exist a subnet $\{ g_l\}$ of $\{ g_j\}$ and isometries $f_m: C_{m}\to X$  such that $g_l\to f_m$ on $C_{m}$ and $f_m(C_{m})=C_{m}$ for all $m$. The last implies
that $\{ g_l\}$ converges to an isometry on $X$, hence $V$ has compact closure.

If $X$ is pseudo-connected the proof of the statement follows directly from Theorem \ref{theorem}.
\end{proof}

\begin{remark}
Note that in Corollary \ref{cor2} we do not require that $X$ is separable as in \cite[Theorem 5.4 and Corollary 6.2]{kechris}. This is not a real improvement since if
$X$ has countably many pseudo-components then it is separable. Indeed, we define a relation on $X$ by $x\mathcal{S} y$ if and only if there exist separable balls
$B_r(x)$ and $B_l(y)$ with $y\in B_r(x)$ and $x\in B_l(y)$. Let $U(x)$ be the equivalence class of $x$ in the transitive closure of the relation $\mathcal{S}$. Then,
each $U(x)$ is a separable clopen subset of $X$ \cite[Lemma 3 in Appendix 2]{kob-nom}. By construction $C_x\subset U(x)$, therefore $X$ is separable.
\end{remark}

\noindent \textbf{Acknowledgements.} We would like to thank the referee for an extremely careful reading of the manuscript. Her/his remarks and comments helped us to
improve considerably the presentation of the paper.

\end{document}